\documentclass{article}

\newtheorem{theorem}{Theorem}

\newtheorem{corollary}[theorem]{Corollary}

\newtheorem{proposition}[theorem]{Proposition}

\newenvironment{proof}[1][Proof]{\noindent\textbf{#1.} }{\ \rule{0.5em}{0.5em}}
\input{tcilatex}
\begin{document}

\date{}
\title{New results on $p$-Bernoulli numbers}
\author{Levent Karg\i n \\
Akseki Vocational School\\
Alanya Alaaddin Keykubat University\\
Antalya TR-07630 Turkey\\
leventkargin48@gmail.com}
\maketitle

\begin{abstract}
We realize that geometric polynomials and $p$-Bernoulli polynomials and
numbers are closely related with an integral representation. Therefore,
using geometric polynomials, we extend some properties of Bernoulli
polynomials and numbers such as recurrence relations, telescopic formula and
Raabe's formula to $p$-Bernoulli polynomials and numbers. In particular
cases of these results, we establish some new results for Bernoulli
polynomials and numbers. Moreover, we evaluate a Faulhaber-type summation in
terms of $p$-Bernoulli polynomials.


\end{abstract}

\section{Introduction}

The Bernoulli polynomials $B_{n}\left( x\right) $ are defined by exponential
generating function 
\begin{equation}
\sum_{n\geq 0}B_{n}\left( x\right) \frac{t^{n}}{n!}=\frac{te^{xt}}{e^{t}-1}%
\text{, }\left\vert t\right\vert <2\pi \text{.}  \label{13}
\end{equation}

In particular, the rational numbers $B_{n}=B_{n}\left( 0\right) $ are called
Bernoulli numbers and have an explicit formula \cite{Graham} 
\[
B_{n}=\sum_{k=0}^{n}\QATOPD\{ \} {n}{k}\frac{\left( -1\right) ^{k}k!}{k+1},
\]%
where $\QATOPD\{ \} {n}{k}$ is the Stirling number of second kind \cite%
{Graham}.

As it is well known, the Bernoulli numbers are considerable importance in
different areas of mathematics such as number theory, combinatorics, special
functions. Moreover, many generalizations and extensions of these numbers
appear in the literature. One of the generalization of the Bernoulli numbers
is $p$-Bernoulli numbers, defined by a three-term recurrence relation \cite%
{Rahmani}%
\begin{equation}
B_{n+1,p}=pB_{n,p}-\frac{\left( p+1\right) ^{2}}{p+2}B_{n,p+1},  \label{16}
\end{equation}%
with the initial condition $B_{0,p}=1$. These numbers also satisfy an
explicit formula 
\[
B_{n,p}=\frac{p+1}{p!}\sum_{k=0}^{n}\QATOPD\{ \} {n+p}{k+p}_{p}\frac{\left(
-1\right) ^{k}\left( k+p\right) !}{k+p+1},
\]%
where $\QATOPD\{ \} {n+p}{k+p}_{p}$ is the $p$-Stirling number of second
kind \cite{Broder}.

As a special case, setting $p=0$ in the above equation gives $B_{n,0}=B_{n}$.

$p$-Bernoulli polynomials which is the polynomial extension of $B_{n,p}$,
are defined by the following convolution%
\begin{equation}
B_{n,p}\left( x\right) =\sum_{k=0}^{n}\binom{n}{k}x^{n-k}B_{k,p}.  \label{18}
\end{equation}%
The first few $p$-Bernoulli polynomials are 
\begin{eqnarray*}
B_{0,p}\left( x\right)  &=&1, \\
B_{1,p}\left( x\right)  &=&x-\frac{1}{p+2}, \\
B_{2,p}\left( x\right)  &=&x^{2}-\frac{2x}{p+2}-\frac{p-1}{\left( p+2\right)
\left( p+3\right) }.
\end{eqnarray*}%
Moreover, these polynomials have integral representations 
\begin{eqnarray}
\dint\limits_{b}^{a}B_{n,p}\left( t\right) dt &=&\frac{B_{n+1,p}\left(
a\right) -B_{n+1,p}\left( b\right) }{n+1},  \label{19} \\
\dint\limits_{0}^{1}B_{n,p}\left( t\right) dt &=&\frac{1}{n+1}\sum_{k=0}^{n}%
\binom{n+1}{k}B_{k,p},  \label{20}
\end{eqnarray}%
a recurrence relation%
\begin{equation}
B_{n,p}\left( x+1\right) -B_{n,p}\left( x\right) =\sum_{k=0}^{n-1}\binom{n}{k%
}B_{k,p}\left( x\right) ,  \label{22}
\end{equation}%
and a three-term recurrence relation%
\begin{equation}
B_{n+1,p}\left( x\right) =\left( x+p\right) B_{n,p}\left( x\right) -\frac{%
\left( p+1\right) ^{2}}{p+2}B_{n,p+1}\left( x\right) .  \label{21}
\end{equation}

In the special case of (\ref{18}) when $x=0$, we obtain $B_{n,p}\left(
0\right) =B_{n,p}$.

Some other properties and applications of $p$-Bernoulli polynomials and
numbers can be found in \cite{Rahmani}.

The main formula of this paper is \cite[p. 361]{Rahmani}%
\[
\frac{1}{p+1}\sum_{n\geq 0}B_{n,p}\frac{t^{n}}{n!}=\dint\limits_{-1}^{0}%
\frac{\left( 1+y\right) ^{p}}{1-y\left( e^{t}-1\right) }dy,\text{ for }p\geq
0.
\]%
Using the generating function of geometric polynomials $w_{n}\left( y\right) 
$ (see Section 2 for details of $w_{n}\left( y\right) $), the above equation
can be written as 
\begin{equation}
\frac{1}{p+1}B_{n,p}=\dint\limits_{-1}^{0}\left( 1+y\right) ^{p}w_{n}\left(
y\right) dy  \label{17}
\end{equation}%
which is the generalization of Keller's identity \cite{KELLER}%
\[
\dint\limits_{-1}^{0}w_{n}\left( y\right) dy=B_{n}.
\]%
Thus, using this integral representation and the properties of geometric
polynomials, we generalize a recurrence relation of Bernoulli numbers to $p$%
-Bernoulli numbers and obtain an explicit formula for $p$-Bernoulli numbers.
Moreover, extending the representation (\ref{17}) to $p$-Bernoulli
polynomials, we give the generalization of the telescopic formula and
Raabe's formula of Bernoulli polynomials for $p$-Bernoulli polynomials.
Thus, as special cases of these results, we get an explicit formula, a
finite summation and a convolution identity\ for Bernoulli polynomials and
numbers. Besides, we evaluate a Faulhaber-type summation in terms of $p$%
-Bernoulli polynomials. 

First, we extend the well known recurrence relation of Bernoulli numbers%
\[
\sum_{k=0}^{n}\binom{n+1}{k}B_{k}=0\text{ \ for }n\geq 1,
\]%
to $p$-Bernoulli numbers in the following theorem.

\begin{theorem}
\label{teo3}For $n\geq 1$ and $p\geq 0$,%
\begin{equation}
\sum_{k=0}^{n}\binom{n+1}{k}B_{k,p}=-pB_{n,p}.  \label{6}
\end{equation}
\end{theorem}

We note that using (\ref{20}) and (\ref{22}) in the above theorem gives us
the following conclusions%
\[
B_{n,p}\left( 1\right) =B_{n,p}-pB_{n-1,p},
\]%
and%
\begin{equation}
\dint\limits_{0}^{1}B_{n,p}\left( t\right) dt=\frac{-pB_{n,p}}{n+1},
\label{7a}
\end{equation}%
respectively. Also, these results are the generalization of the following
well known properties of $B_{n}$%
\[
B_{n}\left( 1\right) =B_{n}\text{ and }\dint\limits_{0}^{1}B_{n}\left(
t\right) dt=0,\text{ \ for }n\geq 1.
\]

The Bernoulli numbers are connected with some well known special numbers 
\cite{Can1, Cenkci2, Merca2, Merca1, Mezo2, Mihioubi2}\textbf{.} Rahmani 
\cite{Rahmani} also gave explicit formulas involving different kind of
special numbers. Now, we obtain a new explicit formula for $B_{n,p}$, and
hence $B_{n}$, in the following theorem.

\begin{theorem}
\label{teo4}For $n\geq 1$ and $p\geq 0$,%
\begin{equation}
B_{n,p}=\left( p+1\right) \sum_{k=1}^{n}\QATOPD\{ \} {n}{k}\frac{\left(
-1\right) ^{k+n}k!}{\left( k+p\right) \left( k+p+1\right) }.  \label{8}
\end{equation}%
When $p=0$ this becomes%
\begin{equation}
B_{n}=\sum_{k=1}^{n}\QATOPD\{ \} {n}{k}\frac{\left( -1\right) ^{k+n}\left(
k-1\right) !}{k+1}.  \label{26}
\end{equation}
\end{theorem}

In order to deal with some properties of $p$-Bernoulli polynomials, we need
to extend the integral representation (\ref{17}) to $B_{n,p}\left( x\right) $%
.

\begin{proposition}
\label{pro2}Let $n$ and $p$ be the non-negative integers. Then we have%
\begin{equation}
\frac{1}{p+1}B_{n,p}\left( x\right) =\dint\limits_{-1}^{0}\left( 1+y\right)
^{p}w_{n}\left( x;y\right) dy,  \label{2}
\end{equation}%
where $w_{n}\left( x;y\right) $ (see Section 2) is two variable geometric
polynomials.
\end{proposition}

One of the important properties of $B_{n}\left( x\right) $ is the telescopic
formula 
\[
B_{n}\left( x+1\right) -B_{n}\left( x\right) =nx^{n-1}.
\]%
From this formula, Bernoulli gave a closed formula for Faulhaber's summation
in terms of Bernoulli polynomials and numbers%
\begin{equation}
\sum_{k=0}^{m}k^{n}=\frac{B_{n+1}\left( m+1\right) -B_{n+1}}{n+1}.  \label{3}
\end{equation}

Now, we want to give an extension of telescopic formula for $p$-Bernoulli
polynomials.

\begin{proposition}
\label{pro1}For any non-negative integer $n$ and $p$,%
\begin{equation}
B_{n,p+1}\left( x+1\right) -B_{n,p+1}\left( x\right) =\frac{p+2}{p+1}\left(
B_{n,p}\left( x+1\right) -x^{n}\right) .  \label{28}
\end{equation}
\end{proposition}

This telescopic formula for $p$-Bernoulli polynomials gives us the
evaluation of finite summation of $p$-Bernoulli polynomials. In particular
case $p=0$, we arrive at a new finite summation involving Bernoulli
polynomials.

\begin{theorem}
\label{teo5}For any non-negative integer $n,m$ and $p$, 
\begin{equation}
\sum_{k=0}^{m}B_{n,p}\left( k+1\right) =\frac{B_{n+1}\left( m+1\right)
-B_{n+1}}{n+1}+\frac{p+1}{p+2}\left( B_{n,p+1}\left( m+1\right)
-B_{n,p+1}\right) .  \label{4}
\end{equation}%
When $p=0$ this becomes%
\begin{equation}
\sum_{k=0}^{m}\left[ B_{n}\left( k+1\right) +nk^{n}\right] =\left(
m+1\right) B_{n}\left( m+1\right) .  \label{27}
\end{equation}
\end{theorem}

Another important identity for Bernoulli polynomials is the Raabe's formula 
\[
m^{n-1}\sum_{k=0}^{m-1}B_{n}\left( x+\frac{k}{m}\right) =B_{n}\left(
mx\right) . 
\]%
Now, we want to extend the Raabe's formula to $p$-Bernoulli polynomials.

\begin{theorem}
\label{teo1}For $m\geq 1$ and $n,p\geq 0$,%
\begin{equation}
m^{n-1}\sum_{k=0}^{m-1}B_{n,p}\left( x+\frac{k}{m}\right) =\left( p+1\right)
B_{n}\left( mx\right) -p\sum_{k=0}^{n}\binom{n}{k}\frac{m^{k}B_{n-k}\left(
mx\right) B_{k,p}}{k+1}.  \label{31}
\end{equation}
\end{theorem}

Using the generating function technique, Chu and Zhou \cite{Chu} give
several convolution identities for Bernoulli numbers. Two of them are the
followings:%
\begin{eqnarray*}
\sum_{k=0}^{n}\binom{n}{k}\frac{B_{k+1}B_{n-k}}{k+1} &=&-B_{n}-B_{n+1}, \\
\sum_{k=0}^{n}\binom{n}{k}\frac{2^{k}B_{k+1}B_{n-k}}{k+1} &=&\frac{%
-B_{n+1}-\left( 2^{n-1}+1\right) B_{n}}{2}.
\end{eqnarray*}

If we set $p=1$ and $x=0$ in (\ref{31}) and use (\ref{16}) and (\ref{17}),
we have a close formula for a generalization of the above equations in the
following corollary.

\begin{corollary}
\label{cor2}For $m\geq 1$ and $n,p\geq 0$,%
\[
\sum_{k=0}^{n}\binom{n}{k}\frac{m^{k}B_{k+1}B_{n-k}}{k+1}=\frac{%
-mB_{n}-B_{n+1}}{m}+m^{n-1}\sum_{k=0}^{m-1}\frac{k}{m}B_{n}\left( \frac{k}{m}%
\right) . 
\]
\end{corollary}

Finally, we evaluate a Faulhaber-type summation in terms of $p$-Bernoulli
polynomials and numbers which generalize the following finite summation \cite%
[p. 18, Eq. 1]{Gould}%
\[
\sum_{k=0}^{n}\frac{\left( -1\right) ^{k}}{\binom{n}{k}}=\frac{n+1}{n+2}%
\left( \left( -1\right) ^{n}+1\right) .
\]

\begin{theorem}
\label{teo2}For $n\geq 1$ and $p\geq 0$, we have%
\[
\sum_{k=0}^{n}\frac{k^{p}\left( -1\right) ^{k}}{\binom{n}{k}}=\frac{n+1}{n+2}%
\left[ \left( -1\right) ^{n+p}B_{p,n+1}\left( -n\right) +B_{p,n+1}\right] . 
\]
\end{theorem}

The summary by sections is as follows: Section 2 is the preliminary section
where we give definitions and known results needed. In Section 3, we derive
a recurrence relation for $p$-Bernoulli and a Raabe-type relation for
geometric polynomials, which we need in the proofs of Theorem \ref{teo1} and
Theorem \ref{teo2}. In Section 4, we give the proofs of the results,
mentioned above.

\section{Preliminaries}

Geometric polynomials are defined by the exponential generating function 
\cite{T}%
\begin{equation}
\frac{1}{1-y\left( e^{t}-1\right) }=\sum_{n=0}^{\infty }w_{n}\left( y\right) 
\frac{t^{n}}{n!}.  \label{14}
\end{equation}%
They have an explicit formula 
\begin{equation}
w_{n}\left( y\right) =y\sum_{k=1}^{n}\QATOPD\{ \} {n}{k}\left( -1\right)
^{n+k}k!\left( y+1\right) ^{k-1},\text{ \ }n>0,  \label{23}
\end{equation}%
and a reflection formula 
\begin{equation}
w_{n}\left( y\right) =\left( -1\right) ^{n}\frac{y}{y+1}w_{n}\left(
-y-1\right) ,\text{ for }n>0,  \label{24}
\end{equation}%
Moreover, these polynomials are related to $p$-Bernoulli numbers with an
integral representation%
\begin{equation}
{\int\limits_{-1}^{0}}y^{p}w_{n}\left( y\right) dy=\left( -1\right) ^{n+p+1}%
\frac{p+1}{p+2}B_{n-1,p+1},\text{ for }n>1,\text{ }p\geq 0.  \label{25}
\end{equation}

See \cite{B, B2, B3, B4, BoyadzhievandDil, Dil1, Kargin1} for other
properties and applications of geometric polynomials.

Two variable geometric polynomials are defined by means of the following
generating function \cite{Kargin1}%
\begin{equation}
\sum_{n=0}^{\infty }w_{n}\left( x;y\right) \frac{t^{n}}{n!}=\frac{e^{xt}}{%
1-y\left( e^{t}-1\right) }.  \label{9}
\end{equation}%
Moreover, they are related to $w_{k}\left( y\right) $ with a convolution%
\begin{equation}
w_{n}\left( x;y\right) =\sum_{k=0}^{n}\binom{n}{k}w_{k}\left( y\right)
x^{n-k},  \label{10}
\end{equation}%
with a special case%
\begin{equation}
w_{n}\left( 0;y\right) =w_{n}\left( y\right) \text{.}  \label{29}
\end{equation}

\section{Some other basic properties}

In this section, in order to use in the proof of Theorem \ref{teo1} and
Theorem \ref{teo2}, we give a recurrence relation for $p$-Bernoulli
polynomials and a Raabe-type formula for two variable geometric polynomials.

For the proof of Theorem \ref{teo1}, we first need the following proposition.

\begin{proposition}
For $n\geq 1$ and $p\geq 0$, we have%
\begin{equation}
p^{2}\sum_{k=1}^{n}\binom{n+1}{k+1}y^{n-k}B_{k,p}=\left( p+1\right)
y^{n+1}+p\left( n+1\right) y^{n}-\left( p+1\right) B_{n+1,p-1}\left(
1+y\right) .  \label{36}
\end{equation}
\end{proposition}

\begin{proof}
From (\ref{18}), we have 
\begin{equation}
\sum_{k=1}^{n}\binom{n}{k}y^{n-k}B_{k,p}\left( x\right) =B_{n,p}\left(
x+y\right) -y^{n}.  \label{11}
\end{equation}%
Let integrate both sides of the above equation with respect to $x$ from $0$
to $1$. Then, using (\ref{7a}), the left hand side of (\ref{11}) becomes%
\begin{eqnarray}
\sum_{k=1}^{n}\binom{n}{k}y^{n-k}{\int\limits_{0}^{1}}B_{k,p}\left( x\right)
dx &=&-p\sum_{k=1}^{n}\binom{n}{k}\frac{y^{n-k}B_{k,p}}{k+1}  \nonumber \\
&=&\frac{-p}{n+1}\sum_{k=1}^{n}\binom{n+1}{k+1}y^{n-k}B_{k,p}.  \label{12}
\end{eqnarray}%
On the other hand, using (\ref{19}) and Proposition \ref{pro1} in the right
hand side of (\ref{11}), we have 
\begin{eqnarray*}
{\int\limits_{0}^{1}}\left[ B_{n,p}\left( x+y\right) -y^{n}\right] dx &=&{%
\int\limits_{y}^{y+1}}B_{n,p}\left( t\right) dt-y^{n}{\int\limits_{0}^{1}dx}
\\
&=&\frac{B_{n+1,p}\left( y+1\right) -B_{n+1,p}\left( y\right) }{n+1}-y^{n} \\
&=&\frac{p+1}{p\left( n+1\right) }\left[ B_{n+1,p-1}\left( y+1\right)
-y^{n+1}\right] -y^{n}.
\end{eqnarray*}%
Combining the above equation with (\ref{12}) gives the desired equation.
\end{proof}

Now, we give the Raabe-type formula for two variable geometric polynomials
in the following proposition. Later, we use it in the proof of Theorem \ref%
{teo2}.

\begin{proposition}
For $m\geq 1$ and $n,p\geq 0$,%
\begin{equation}
m^{n-1}\sum_{k=0}^{m-1}w_{n-1}\left( x+\frac{k}{m},y\right) =\frac{1}{ny}%
\sum_{k=1}^{n}\binom{n}{k}m^{k}B_{n-k}\left( mx\right) w_{k}\left( y\right) .
\label{15}
\end{equation}
\end{proposition}

\begin{proof}
Using (\ref{9}) and the identity%
\[
\sum_{k=0}^{m-1}x^{k}=\frac{x^{m}-1}{x-1},
\]%
we have 
\begin{eqnarray*}
\sum_{n=0}^{\infty }\frac{t^{n}}{n!}\sum_{k=0}^{m-1}w_{n}\left( x+\frac{k}{m}%
,y\right)  &=&\sum_{k=0}^{m-1}\sum_{n=0}^{\infty }w_{n}\left( x+\frac{k}{m}%
,y\right) \frac{t^{n}}{n!} \\
&=&\frac{1}{1-y\left( e^{t}-1\right) }\sum_{k=0}^{m-1}e^{\left( x+\frac{k}{m}%
\right) t} \\
&=&\frac{e^{xt}}{1-y\left( e^{t}-1\right) }\frac{e^{t}-1}{e^{t/m}-1} \\
&=&\frac{e^{xt}}{e^{t/m}-1}\frac{1}{1-y\left( e^{t}-1\right) }\frac{\left(
1-\left( 1-y\left( e^{t}-1\right) \right) \right) }{y} \\
&=&\frac{1}{y\left( t/m\right) }\left( \frac{\left( t/m\right) e^{xt}}{%
e^{t/m}-1}\frac{1}{1-y\left( e^{t}-1\right) }-\frac{\left( t/m\right) e^{xt}%
}{e^{t/m}-1}\right) .
\end{eqnarray*}%
From (\ref{13}) and (\ref{14}), the above equation can be written as%
\[
\frac{yn}{m}\sum_{n=1}^{\infty }\frac{t^{n}}{n!}\sum_{k=0}^{m-1}w_{n-1}%
\left( x+\frac{k}{m},y\right) =\sum_{n=1}^{\infty }\frac{t^{n}}{n!}\left[
\sum_{k=1}^{n}\binom{n}{k}\frac{B_{n-k}\left( mx\right) w_{k}\left( y\right) 
}{m^{n-k}}-\frac{B_{n}\left( mx\right) }{m^{n}}\right] .
\]%
Finally, comparing the coefficients of $\frac{t^{n}}{n!}$ in the both sides
of the above equation, we get (\ref{15}).
\end{proof}

\section{Proofs}

In this section, we give the proofs of all results mentioned in Section 1.

\begin{proof}[Proof of Theorem \protect\ref{teo3}]
Using (\ref{24}) in the following equation \cite[Proposition 15]%
{BoyadzhievandDil}, we have 
\begin{eqnarray*}
\sum_{k=0}^{n}\binom{n}{k}w_{k}\left( y\right) &=&\frac{1+y}{y}w_{n}\left(
y\right) \\
&=&\left( -1\right) ^{n}w_{n}\left( -y-1\right) .
\end{eqnarray*}%
Multiplying both sides of the above equation by $\left( 1+y\right) ^{p}$,
integrating it with respect to $y$ from $-1$ to $0$ and using (\ref{17}) and
(\ref{25}), we achieve 
\begin{eqnarray*}
\frac{1}{p+1}\sum_{k=0}^{n}\binom{n}{k}B_{k,p} &=&\left( -1\right) ^{n}{%
\int\limits_{-1}^{0}}\left( {1+}y\right) ^{p}w_{n}\left( -y-1\right) dy \\
&=&\left( -1\right) ^{n+p}{\int\limits_{-1}^{0}}x^{p}w_{n}\left( x\right) dx
\\
&=&-\frac{p+1}{p+2}B_{n-1,p+1}.
\end{eqnarray*}%
Finally, using (\ref{16}) gives the desired equation.
\end{proof}

\begin{proof}[Proof of Theorem \protect\ref{teo4}]
Multiplying both sides of (\ref{23}) by $\left( 1+y\right) ^{p}$,
integrating it with respect to $y$ from $-1$ to $0$ and using (\ref{17}), we
have%
\begin{eqnarray*}
\frac{1}{p+1}B_{n,p} &=&\sum_{k=1}^{n}\QATOPD\{ \} {n}{k}\left( -1\right)
^{n+k}k!\int\limits_{-1}^{0}y\left( y+1\right) ^{p+k-1}dy \\
&=&\sum_{k=1}^{n}\QATOPD\{ \} {n}{k}\left( -1\right)
^{n+k+1}k!\int\limits_{0}^{1}\left( 1-x\right) x^{p+k-1}dx.
\end{eqnarray*}%
Finally, using the well known relation of Beta function 
\begin{equation}
B\left( x,y\right) =\int\limits_{0}^{1}\left( 1-t\right) ^{x-1}t^{y-1}dt=%
\frac{\left( x-1\right) !\left( y-1\right) !}{\left( x+y-1\right) !},
\label{34}
\end{equation}%
where $x,y=1,2,3,\cdots $, we obtain (\ref{8}).
\end{proof}

\begin{proof}[Proof of Proposition \protect\ref{pro2}]
Using the equations (\ref{17}) and (\ref{18}) in (\ref{10}), we have%
\begin{eqnarray*}
{\int\limits_{-1}^{0}}\left( {1+}y\right) ^{p}w_{n}\left( x;y\right) dy
&=&\sum_{k=0}^{n}\binom{n}{k}x^{n-k}{\int\limits_{-1}^{0}}\left( {1+}%
y\right) ^{p}w_{k}\left( y\right) dy \\
&=&\sum_{k=0}^{n}\binom{n}{k}x^{n-k}B_{k,p} \\
&=&\frac{1}{p+1}B_{n,p}.
\end{eqnarray*}
\end{proof}

\begin{proof}[Proof of Proposition \protect\ref{pro1}]
The two variable geometric polynomials have \cite[Eq. 14]{Kargin1}%
\[
w_{n}\left( x+1;y\right) -w_{n}\left( x;y\right) =\frac{1}{1+y}\left(
w_{n}\left( x+1,y\right) -x^{n}\right) .
\]%
Multiplying both sides of the above equation by $\left( 1+y\right) ^{p+1}$,
integrating it with respect to $y$ from $-1$ to $0$ and using (\ref{17})
yield (\ref{28}).
\end{proof}

\begin{proof}[Proof of Theorem \protect\ref{teo5}]
Replacing $x$ with $k$ in (\ref{28}) and summing over $k$ from $0$ to $m$,
we obtain%
\begin{eqnarray}
\frac{p+2}{p+1}\sum_{k=0}^{m}B_{n,p}\left( k+1\right) -\sum_{k=0}^{m}k^{n}
&=&\sum_{k=0}^{m}\left( B_{n,p+1}\left( k+1\right) -B_{n,p+1}\left( k\right)
\right)  \nonumber \\
&=&B_{n,p+1}\left( m+1\right) -B_{n,p+1}.  \label{30}
\end{eqnarray}%
If we use Bernoulli's well known identity for Faulhaber summation 
\[
\sum_{k=0}^{m}k^{n}=\frac{B_{n+1}\left( m+1\right) +B_{n+1}}{n+1}, 
\]%
in the second part of the left hand side of (\ref{30}), we arrive at the
first part of theorem.

For the second part of the theorem, if we use (\ref{16}) and (\ref{21}) for $%
p=0$, (\ref{4}) becomes%
\begin{eqnarray*}
\sum_{k=0}^{m}B_{n}\left( k+1\right) &=&\left( m+1\right) B_{n}\left(
m+1\right) -n\frac{B_{n+1}\left( m+1\right) +B_{n+1}}{n+1} \\
&=&\left( m+1\right) B_{n}\left( m+1\right) -n\sum_{k=0}^{m}k^{n}.
\end{eqnarray*}%
Then, we have (\ref{27}).
\end{proof}

\begin{proof}[Proof of Theorem \protect\ref{teo1}]
Multiplying both sides of (\ref{15}) by $\left( 1+y\right) ^{p+1}$ and using
(\ref{24}), we have%
\begin{eqnarray*}
&&m^{n-1}\sum_{k=0}^{m-1}\left( 1+y\right) ^{p+1}w_{n-1}\left( x+\frac{k}{m}%
,y\right) \\
&&\qquad \quad =\frac{1}{n}\sum_{k=1}^{n}\binom{n}{k}m^{k}B_{n-k}\left(
mx\right) \frac{\left( 1+y\right) ^{p+1}}{y}w_{k}\left( y\right) \\
&&\qquad \quad =\frac{1}{n}\sum_{k=1}^{n}\binom{n}{k}m^{k}B_{n-k}\left(
mx\right) \left( 1+y\right) ^{p}\left( -1\right) ^{k}w_{k}\left( -y-1\right)
\\
&&\qquad \quad =mB_{n-1}\left( mx\right) \left( 1+y\right) ^{p+1}+\frac{1}{n}%
\sum_{k=2}^{n}\binom{n}{k}m^{k}B_{n-k}\left( mx\right) \left( -1\right)
^{k}\left( 1+y\right) ^{p}w_{k}\left( -y-1\right) .
\end{eqnarray*}%
Integrating the above equation with respect to $y$ from $-1$ to $0$ and
using (\ref{17}) and (\ref{25}), we obtain%
\begin{eqnarray*}
&&\frac{m^{n-1}}{p+2}\sum_{k=0}^{m-1}B_{n-1,p+1}\left( x+\frac{k}{m}\right)
\\
&&\qquad \qquad =\frac{mB_{n-1}\left( mx\right) }{p+2}+\frac{1}{n}%
\sum_{k=2}^{n}\binom{n}{k}m^{k}B_{n-k}\left( mx\right) \left( -1\right)
^{k}\int\limits_{-1}^{0}\left( 1+y\right) ^{p}w_{k}\left( -y-1\right) dy \\
&&\qquad \qquad =\frac{mB_{n-1}\left( mx\right) }{p+2}+\frac{1}{n}%
\sum_{k=2}^{n}\binom{n}{k}m^{k}B_{n-k}\left( mx\right) \left( -1\right)
^{k+p}\int\limits_{-1}^{0}x^{p}w_{k}\left( x\right) dx \\
&&\qquad \qquad =\frac{mB_{n-1}\left( mx\right) }{p+2}-\frac{p+1}{n\left(
p+2\right) }\sum_{k=2}^{n}\binom{n}{k}m^{k}B_{n-k}\left( mx\right)
B_{k-1,p+1} \\
&&\qquad \qquad =mB_{n-1}\left( mx\right) -\frac{p+1}{n\left( p+2\right) }%
\sum_{k=1}^{n}\binom{n}{k}m^{k}B_{n-k}\left( mx\right) B_{k-1,p+1}.
\end{eqnarray*}%
Replacing $p+1$ with $p$ and $n$ with $n+1$, the above equation can be
rewritten as%
\begin{eqnarray*}
m^{n-1}\sum_{k=0}^{m-1}B_{n,p}\left( x+\frac{k}{m}\right) &=&\left(
p+1\right) B_{n}\left( mx\right) -\frac{p}{n+1}\sum_{k=1}^{n+1}\binom{n+1}{k}%
m^{k-1}B_{n+1-k}\left( mx\right) B_{k-1,p} \\
&=&\left( p+1\right) B_{n}\left( mx\right) -\frac{p}{n+1}\sum_{k=0}^{n}%
\binom{n+1}{k+1}m^{k}B_{n-k}\left( mx\right) B_{k,p} \\
&=&\left( p+1\right) B_{n}\left( mx\right) -p\sum_{k=0}^{n}\binom{n}{k}\frac{%
m^{k}B_{n-k}\left( mx\right) B_{k,p}}{k+1}.
\end{eqnarray*}
\end{proof}

\begin{proof}[Proof of Theorem \protect\ref{teo2}]
We have an arithmetic-geometric progression \cite{DeBruyn, WangandHsu}%
\begin{equation}
\sum_{k=0}^{n}k^{p}y^{k}=\frac{A_{p}\left( y\right) }{\left( 1-y\right)
^{p+1}}-y^{n+1}\sum_{k=0}^{p}\binom{p}{k}\left( n+1\right) ^{p-k}\frac{%
A_{k}\left( y\right) }{\left( 1-y\right) ^{k+1}},  \label{32}
\end{equation}%
where $A_{n}\left( y\right) $ is the Eulerian polynomial of degree $n$ \cite%
{COMTET}. These polynomials are closely related to the geometric polynomials
with relation \cite[Eq. 3.18]{B}%
\[
A_{n}\left( y\right) =\left( 1-y\right) ^{n}w_{n}\left( y\right) .
\]%
If we multiply both side of (\ref{32}) and use this relation, (\ref{32}) can
be rewritten as 
\[
\sum_{k=0}^{n}k^{p}y^{k}\left( 1+y\right) ^{n-k}=\left( 1+y\right)
^{n+1}w_{p}\left( y\right) -y^{n+1}\sum_{k=0}^{p}\binom{p}{k}\left(
n+1\right) ^{p-k}w_{k}\left( y\right) .
\]%
Integrating both sides of the above equation with respect $y$ from $-1$ to $0
$, we have%
\begin{eqnarray}
&&\sum_{k=0}^{n}k^{p}\int\limits_{-1}^{0}y^{k}\left( 1+y\right) ^{n-k}dy 
\nonumber \\
&&\quad =\int\limits_{-1}^{0}\left( 1+y\right) ^{n+1}w_{p}\left( y\right)
dy-\left( n+1\right) ^{p}\int\limits_{-1}^{0}y^{n+1}dy-p\left( n+1\right)
^{p-1}\int\limits_{-1}^{0}y^{n+2}dy  \nonumber \\
&&\qquad -\sum_{k=2}^{p}\binom{p}{k}\left( n+1\right)
^{p-k}\int\limits_{-1}^{0}y^{n+1}w_{k}\left( y\right) .  \label{33}
\end{eqnarray}%
Using (\ref{34}) in the left hand side of (\ref{33}) yields%
\[
\sum_{k=0}^{n}k^{p}\int\limits_{-1}^{0}y^{k}\left( 1+y\right) ^{n-k}dy=\frac{%
1}{n+1}\sum_{k=0}^{n}\frac{\left( -1\right) ^{k}}{\binom{n}{k}}.
\]%
From (\ref{17}), the first integral in the right hand side of (\ref{33})
becomes%
\[
\int\limits_{-1}^{0}\left( 1+y\right) ^{n+1}w_{p}\left( y\right) dy=\frac{1}{%
n+2}B_{p,n+1}.
\]%
The second and third integrals in the right hand side can be evaluated
easily. For the last integral in right hand side of (\ref{33}), if we use (%
\ref{36}) and (\ref{25}), we have 
\begin{eqnarray*}
&&\sum_{k=2}^{p}\binom{p}{k}\left( n+1\right)
^{p-k}\int\limits_{-1}^{0}y^{n+1}w_{k}\left( y\right)  \\
&&\qquad \qquad =\frac{\left( -1\right) ^{n+p}\left( n+1\right) ^{p}}{n+2}+%
\frac{\left( -1\right) ^{n+1}p\left( n+1\right) ^{p-1}}{n+3}-\frac{\left(
-1\right) ^{n+p}B_{p,n+1}\left( -n\right) }{n+2}.
\end{eqnarray*}%
Finally, combining all these evaluated integrals give the desired equation.
\end{proof}

\end{document}